\documentclass{article}
\usepackage{amsmath}
\usepackage{amsfonts}
\usepackage{amssymb}
\usepackage{amsthm}
\usepackage{geometry}
\usepackage{enumerate}
\usepackage{bbm}
\usepackage{url}
\usepackage[usenames,dvipsnames]{color}
\usepackage{algorithm, algorithmic}

\usepackage[colorlinks=false,,linktoc=allbookmarksopen=true,linktocpage,pdftex]{hyperref}
\usepackage[pdftex]{bookmark}

\newtheorem{theorem}{Theorem}[section]

\newtheorem{remark}[theorem]{Remark}
\newtheorem{lemma}[theorem]{Lemma}

\newcommand{\cV}{S}
\newcommand{\cI}{I}
\newcommand{\cv}{x}
\newcommand{\cu}{y}

\newcommand{\cA}{\mathcal{A}}
\newcommand{\cB}{\mathcal{B}}
\newcommand{\cM}{M}
\newcommand{\cQ}{S}
\newcommand{\cT}{\mathcal{T}}

\newcommand{\ctQ}{\tilde{\mathcal{Q}}}

\newcommand{\cxi}{\xi}

\newcommand{\cS}{\mathcal{S}}

\newcommand{\cPr}{\mathbb{P}}
\newcommand{\E}{\mathbb{E}}

\newcommand{\NI}{Y}

\newcommand\set[1]{\ensuremath{\{#1\}}}

\newcommand\bigpar[1]{\bigl(#1\bigr)}
\newcommand\Bigpar[1]{\Bigl(#1\Bigr)}
\newcommand\biggpar[1]{\biggl(#1\biggr)}

\newcommand\bigsqpar[1]{\bigl[#1\bigr]}
\newcommand\Bigsqpar[1]{\Bigl[#1\Bigr]}
\newcommand\biggsqpar[1]{\biggl[#1\biggr]}

\newcommand\bigcpar[1]{\bigl\{#1\bigr\}}

\newcommand\bigabs[1]{\bigl|#1\bigr|}

\def\rompar(#1){\textup(#1\textup)}    

\def\xexp(#1){e^{#1}}
\newcommand\ceil[1]{\lceil#1\rceil}

\newcommand\bigfloor[1]{\bigl\lfloor#1\bigr\rfloor}

\newcommand{\refT}[1]{Theorem~\ref{#1}}

\newcommand{\refL}[1]{Lemma~\ref{#1}}
\newcommand{\refR}[1]{Remark~\ref{#1}}
\newcommand{\refS}[1]{Section~\ref{#1}}

\newcommand{\indic}[1]{\mathbbm{1}_{\{{#1}\}}}

\newcommand{\Bin}{\operatorname{Bin}}

\newcommand{\cpD}{D^{\frac{1}{r-1}}} 
\newcommand{\cmD}{D^{-\frac{1}{r-1}}}

\newcommand{\cN}{\mathcal{N}} 

\newcommand{\cF}{\mathcal{F}} 

\newcommand{\tm}{m} 

\newcommand{\cGo}{\mathcal{G}} 
\newcommand{\err}{e}
\newcommand{\cf}{f}

\newcommand{\lbc}{c}
\newcommand{\ubC}{C}
\newcommand{\biC}{C'}

\renewcommand\P{\operatorname{\mathbb P{}}}

\renewcommand{\emptyset}{\varnothing}
\newcommand\noproof{\qed}

\AtBeginDocument{%
	\def\MR#1{}
}

\let\OLDthebibliography\thebibliography
\renewcommand\thebibliography[1]{
  \OLDthebibliography{#1}
  \setlength{\parskip}{0pt}
  \setlength{\itemsep}{0pt plus 0.3ex}
}

\title{On the power of random greedy algorithms} 
\author{He Guo%
\thanks{Faculty of Mathematics, Technion, Haifa~32000, Israel. E-mail: {\tt hguo@campus.technion.ac.il}.}
\ and Lutz Warnke%
	\thanks{Department of Mathematics, University of California, San Diego, La Jolla CA~92093, USA. 
E-mail: {\tt lwarnke@ucsd.edu}. 
Supported by NSF~grant DMS-1703516, NSF~CAREER grant~DMS-1945481, and a Sloan Research Fellowship.}
}
\date{April~15, 2021; revised April~3, 2022}
\begin{document}
\maketitle
\begin{abstract}
In this paper we solve two problems of Esperet, Kang and Thomass\'{e} as well as~Li 
concerning \mbox{(i)~induced} bipartite subgraphs in triangle-free graphs and (ii)~van der Waerden numbers. 
Each time \mbox{random~greedy~algorithms} allow us to go beyond the \mbox{Lov\'asz~Local~Lemma} or \mbox{alteration~method} used in previous work, 
illustrating the power~of the algorithmic~approach to the probabilistic~method. 
\end{abstract}

\section{Introduction}
The probabilistic method is a widely used tool for proving the existence of hard-to-construct mathematical objects with certain desirable properties: 
it works by showing that a randomly chosen object has the desired properties with non-zero probability.
In classical textbook approaches to the probabilistic method, 
the underlying random objects are typically generated in a \emph{static} way, 
e.g., by choosing a graph uniformly at random from a prescribed class of graphs,
or by independently including each possible edge.

In this paper we illustrate the power of the algorithmic approach to the probabilistic method, 
where the random objects are generated step-by-step in a \emph{dynamic} way using a randomized algorithm. 
To this end we consider two examples from graph theory and additive combinatorics, 
and show that each time 
random greedy algorithms allow us to 
go beyond classical applications of the probabilistic~method,  
i.e., prove existence of mathematical objects with better properties. 
These algorithmic improvements are key for (i)~resolving a problem of Esperet, Kang and Thomass\'{e}~\cite{esperet2019separation}, 
and (ii)~answering a question of Li~\cite{yushenglitalk},  
see Theorems~\ref{thm:fnd}~and~\ref{thm:van}. 

For the two combinatorial examples considered in this paper, 
previous work used the probabilistic method to show that static random objects can avoid certain forbidden substructures, 
while maintaining other desired pseudo-random properties. 
Our technical results show that random greedy algorithms, 
which by construction avoid these forbidden substructures, 
create objects with \mbox{superior} pseudo-random properties, see Theorems~\ref{thm:semi} and~\ref{thm:AP}. 
With the benefit of hindsight, earlier work 
of R\"{o}dl~\cite{rodl1985packing}, Kahn~\cite{kahn1996asymptotically}, Wormald~\cite{wormald1995}, Spencer~\cite{spencer1995}, Kim~\cite{kim1995ramsey}, Bohman~\cite{bohman2009triangle}, and others~\mbox{\cite{bennett2016note,bohman2019large,Glock2020conjecture}} 
can be interpreted similarly. 
This paper thus reveals the following emerging algorithmic~paradigm:  
one can often take proofs based on classical probabilistic~method arguments, 
and obtain improvements by using an algorithmic approach to the probabilistic~method.

\subsection{Induced bipartite subgraphs in triangle-free graphs}\label{sec:TrInduced}
Our first example is from extremal graph theory, concerning a local refinement of the famous Max Cut problem. 
Here the history starts in~1988, when Erd\H{o}s, Faudree, Pach and Spencer~\cite{erdos1988make} 
introduced the problem of searching for large induced bipartite subgraphs in triangle-free graphs. 
Around~2018 Esperet, Kang and Thomass\'{e}~\cite{esperet2019separation} further refined this problem, focusing on induced bipartite subgraphs with large minimum degree.
More precisely, for fixed~$\eta \in (0,1)$ they asked to determine the behavior of the parameter~$f_\eta(n)$, 
which is defined as the maximum~$f$ such that every $n$-vertex triangle-free graph with minimum degree at least~$n^{\eta}$ contains an induced bipartite subgraph with minimum degree at least~$f$. 
Recent results of Kwan, Letzter, Sudakov and Tran~\cite{kwan2020dense} and Cames van Batenburg, de Joannis de Verclos, Kang and Pirot~\cite{van2020bipartite} show~that 
\begin{equation}\label{eq:feta}
f_\eta(n) = \Theta\bigl(\max\bigl\{\log n,n^{2\eta-1}\bigr\}\bigr) \qquad \text{ for fixed~$\eta \in (0,1) \setminus (1/2,2/3]$,}
\end{equation}
and also determine $f_\eta(n)$ up to logarithmic factors in the remaining range~$\eta \in (1/2,2/3]$. 
Illustrating the  conceptual punchline of this paper, 
we use a `dynamic' randomized greedy algorithm to improve existing upper bound constructions~\cite{kwan2020dense,van2020bipartite}, 
which were based on classical probabilistic method tools applied to the binomial random graph~$G(n,p)$. 
This algorithmic improvement allows us to close the logarithmic gap for~$\eta \in (1/2,2/3]$, and determine the order of magnitude of $f_\eta(n)$ for any fixed~$\eta \in (0,1)$. 
The following result in particular resolves~\cite[Problem 4.1]{esperet2019separation} of Esperet, Kang and Thomass\'{e} up to constant~factors.
\begin{theorem}\label{thm:fnd}
For fixed~$\eta \in (0,1)$, we have $f_\eta(n)=\Theta(\max\{\log n,n^{2\eta-1}\})$.  
\end{theorem}
In comparison to the previous upper bounds~\cite{kwan2020dense,van2020bipartite} based on 
the probabilistic analysis of~$G(n,p)$ via the alteration method or the Lov\'asz Local Lemma,  
our key improvement stems from the fact that via the so-called semi-random triangle-free process we are able to algorithmically construct pseudo-random triangle-free graphs with higher edge density (see~\refT{thm:semi}), 
confirming speculations from~\mbox{\cite[Section~4]{esperet2019separation}} and~\mbox{\cite[Section~3]{van2020bipartite}}. 
%

\subsection{Van der Waerden numbers}\label{sec:intr:vdW}
Our second example is from additive combinatorics, concerning a well-known Ramsey-type parameter for arithmetic progressions. 
The van der Waerden number~$W(r,k)$ is defined as the smallest integer~$n$ such that every red and blue coloring of numbers in~$[n]:=\{1,2,\dots, n\}$ contains either a monochromatic red 
 $r$-term arithmetic progression ($r$-AP) or a monochromatic blue $k$-AP. 
The celebrated van der Waerden's theorem guarantees that~$W(r,k)$ is finite for all integers $r,k \ge 2$, 
making it a natural and interesting problem to determine the asymptotic behavior of~$W(r,k)$, see~\cite{Gowers2001new,Green2002AP}. 
The `off-diagonal' case, where~$r \ge 3$~is fixed and~$k$~tends to infinity, was of particular interest to Graham 
(note that~$W(2,k)=\Theta(k)$ holds trivially). 
Indeed, in the mid~2000s Graham conjectured that~$W(3,k) \le k^{O(1)}$, and mentioned that numerical evidence suggests~$W(3,k)= k^{2+o(1)}$, see~\mbox{\cite{graham2006growth,graham2015rudiments,Green2010note}}.
Around~2015 Graham even started offering a~\$250 reward for his conjecture, see~{\cite[p.~19]{graham2015rudiments}}. 
In terms of lower bounds, in~2008 Li and Shu~\cite{li2010lower} showed that 
\[
W(r,k)=\Omega\bigpar{(k/\log k)^{r-1}} \qquad \text{ for fixed~$r \ge 3$,}
\] 
by applying the Lov\'asz Local Lemma to a random subset of the integers~$[n]$. 
Subsequently, Li raised in~2009 the natural question~\cite{yushenglitalk} whether this probabilistic lower bound can be improved via a randomized greedy algorithm that `dynamically' constructs an $r$-AP free subset of the integers~$[n]$.
The proof of the following theorem answers Li's question affirmatively, see also Sections~\ref{sec:techn:vdW}~and~\ref{sec:van}. 
\begin{theorem}\label{thm:van}
For fixed~$r\ge 3$, we have $W(r,k)=\Omega\big( k^{r-1}/(\log k)^{r-2}\big)$. 
\end{theorem}
This result was announced in October~2020, see~\cite{guovdwtalk}.
While preparing this paper, 
Green~\cite{green2021new} made a breakthrough
and showed ${W(3,k) \ge k^{(\log k)^{1/3-o(1)}}}$ using very different techniques, 
which in view of ${W(r,k) \ge W(3,k)}$
disproves the earlier belief that ${W(r,k)=k^{O(1)}}$ for fixed~$r \ge 3$. 
The best known upper bound $W(3,k)\le \exp\bigl(k^{1-\Omega(1)}\bigr)$ was obtained by Schoen~\cite{schoen2020subexponential}.

\subsection{Organization}
In \refS{sex:techn} we state our main technical results, which will imply \refT{thm:fnd} and~\ref{thm:van} for induced bipartite subgraphs and van der Waerden numbers, respectively. 
In Sections~\ref{sec:gnd} and~\ref{sec:van} we then prove these technical results using an algorithmic approach to the probabilistic method, 
i.e., by analyzing randomized algorithms that construct pseudo-random triangle-free graphs and $r$-AP free subsets of the integers, respectively.

\subsection{Main technical results}\label{sex:techn}

\subsubsection{Construction of pseudo-random triangle-free graphs}\label{sec:techn:TrInduced}
%
To prove the upper bound on the parameter~$f_\eta(n)$ claimed by \refT{thm:fnd} for~$\eta \in (1/2,2/3]$, 
our strategy is to construct a pseudo-random triangle-free graph~$G_n$ with $\Theta(n)$ vertices, 
where pseudo-randomness will intuitively ensure the desired minimum degree properties (in suitable constructions that are based on~$G_n$). 
Following the conceptual punchline of this paper, we shall construct the desired graph~$G_n$ using a semi-random variant of the triangle-free process, 
which is a randomized greedy algorithm that sequentially adds more edges to~$G_n$ without creating a triangle, see~\refS{sec:gnd} for the full details. 
This algorithmic approach to the probabilistic method is key for obtaining our improved upper bound on~$f_\eta(n)$ via the following auxiliary result, 
since earlier approaches based on the binomial random graph~$G(n,p)$ 
were only able to prove a weaker version of \refT{thm:semi}, where the triangle-free graph~$G_n$ is sparser, i.e., with minimum and maximum degree bounds~$\delta(G_n),\Delta(G_n)=\Theta(\sqrt{n})$; see~\cite[Lemma~5.1]{kwan2020dense} and~\cite[Theorem~3.1]{van2020bipartite}. 
Note that we recover these earlier results, by sampling each edge of the graph~$G_n$ from~\refT{thm:semi} independently with probability~$1/\sqrt{\log n}$.
\begin{theorem}\label{thm:semi}
There are constants~$\lbc,\ubC,\biC>0$ such that for any~$0<\beta<1/14$ the following holds for any integer~$n \ge n_0=n_0(\beta)$. 
There exists a triangle-free graph~$G_n$ with $v(G_n) \in [n/3,n]$ vertices, 
	\begin{gather}\label{eq:maxmindegree}
	\lbc\sqrt{\beta n\log n}\le 	\delta(G_n)\le \Delta(G_n)\le \ubC\sqrt{\beta n\log n},
	\end{gather}
and the property that any induced bipartite subgraph~$F \subseteq G_n$ has minimum degree~$\delta(F) \le \biC\log n$. 
\end{theorem}
We defer the proof of~\refT{thm:semi} to \refS{sec:gnd}: it is based on a careful refinement\footnote{In the proof of \refT{thm:semi} we mainly use the semi-random variant instead of the standard triangle-free process~\cite{bohman2009triangle,pontiveros2013triangle,bohman2013dynamic} for technical convenience, since in~\refS{sec:gnd} this leads to conceptually simpler proofs of certain pseudo-random properties (i.e., where the necessary refinements of existing proofs require less technical work for the semi-random variant).}  
of the semi-random triangle-free process analysis of Guo and Warnke~\cite{guo2020packing}. 
Using~\refT{thm:semi} we shall in fact establish improved bounds for the more general parameter~$g(n,d)$, 
which denotes the maximum~$g$ such that every $n$-vertex triangle-free graph with minimum degree at least~$d$ contains an induced bipartite subgraph with minimum degree at least~$g$. 
Extending~\cite{kwan2020dense,van2020bipartite}, the following result establishes a phase transition of~$g(n,d)$ when the minimum degree~$d$ is around~$\sqrt{n \log n}$, 
and it also implies \refT{thm:fnd} since~$f_\eta(n)=g(n,n^\eta)$.
\begin{theorem}\label{thm:gnd}
For any fixed~$\gamma \in (0,1)$, we have $g(n,d)=\Theta\big(\max\{\log d,d^2/n\}\big)$ for all~$n^{\gamma}\le d\le n/2$.
\end{theorem}
Similar to~$f_\eta(n)=g(n,n^\eta)$, 
the cases~${n^{\gamma} \le d \le \sqrt{n}}$ and~${n^{2/3} \le d \le n/2}$ of \refT{thm:gnd} follow from~\cite{kwan2020dense}.
Furthermore, for~${\sqrt{n} \le d \le n^{2/3}}$ we obtain~$g(n,d)={\Omega(\max\{\log d,d^2/n\})}$ by combining~\mbox{\cite[Theorem~1.3]{kwan2020dense}} 
with the fact that~$g(n,d)$ is monotone increasing in~$d$. 
We now close the gap for~${\sqrt{n} \le d \le n^{2/3}}$ 
by mimicking the upper bound constructions from~\cite{kwan2020dense,van2020bipartite} 
using the semi-random triangle-free process based graphs~$G_n$ from \refT{thm:semi},
which have better degree properties than the~$G(n,p)$ based graphs used in~\cite{kwan2020dense,van2020bipartite}. 
\begin{proof}[Proof of \refT{thm:gnd} based on~\refT{thm:semi}]
Writing~$c,C'>0$ for the constants of~\refT{thm:semi}, let~${\beta:=10^{-2}}$ and~${A:=c\sqrt{\beta}/3}$. 
As discussed, it suffices to prove $g(n,d)={O(\max\{\log d, d^2/n\})}$ 
for~${\sqrt{n} \le d \le n^{2/3}}$. 
To this end we may always assume that~$n$ is sufficiently large (whenever necessary), as usual.  

We start with the case~${\sqrt{n} \le d\le A \sqrt{n\log n}}$, 
where we set~$\alpha := {2/(\lbc^2 \beta)}$ and~${n':=\lceil \alpha d^2/\log n\rceil}$. 
Note that~${n^{2/3} \ll n' \le \lceil{\alpha A^2 n}\rceil \le n/2}$ and~$n' \ll d^2$. 
By taking the disjoint union of~$\lfloor n/n'\rfloor$ copies of~$G_{n'}$, 
we obtain a triangle-free graph~$H_n$ with $v(H_n) = \lfloor n/n'\rfloor \cdot v(G_{n'}) \in [n/6,n]$ vertices and minimum degree
\begin{align*}
\delta(H_n) = \delta(G_{n'}) \ge c \sqrt{\beta n' \log n'} \ge \sqrt{c^2\beta \alpha \cdot  d^2 \cdot 2/3} > d. 
\end{align*}
Furthermore, every induced bipartite subgraph~$F \subseteq H_n$ is a disjoint union of induced bipartite subgraphs from copies of $G_{n'}$ and thus has minimum degree at most~$\delta(F) \le \biC\log n' \le 2 \biC \log d$.
By `blowing up' each vertex of~$H_n$ into an independent set of suitable sizes between one and six
(i.e., after replacing each vertex of~$H_n$ by an independent set, we add a complete bipartite graph between every pair of independent sets that correspond to an edge in~$H_n$),  
we thus obtain an $n$-vertex triangle-free graph~$G_{n,d}$ with~${\delta(G_{n,d}) \ge \delta(H_n) \ge d}$, 
where furthermore every induced bipartite subgraph~$F \subseteq G_{n,d}$ has minimum degree at most~${\delta(F) \le 6 \cdot 2\biC \log d}$ (by analogous disjoint reasoning as before), 
establishing that~$g(n,d)=O(\log d)$. 

Finally, in the remaining case~${A \sqrt{n\log n} \le d\le n^{2/3}}$ 
we set~$\alpha := {\lbc^2 \beta/18}$ and~${n':=\lfloor \alpha (n/d)^2\log n\rfloor}$. 
Note that~${n^{2/3} \ll n' \le \alpha n/A^2 \le n/2}$. 
By `blowing up' each vertex of~$G_{n'}$ into an independent set of size~$\lfloor n/n'\rfloor$, 
we obtain a triangle-free graph~$H_n$ with $v(H_n) = \lfloor n/n'\rfloor \cdot v(G_{n'}) \in [n/6,n]$ vertices  
and minimum degree 
\begin{align*}
\delta(H_n) = \lfloor n/n'\rfloor \cdot \delta(G_{n'})\ge \frac{n}{2n'}\cdot \lbc\sqrt{\beta n'\log n'} 
\ge \sqrt{\frac{\lbc^2\beta n^2 \log(n^{2/3})}{4n'}} 
\ge \sqrt{\frac{\lbc^2 \beta \cdot d^2 \cdot 2/3}{4\alpha}} > d .
\end{align*}
Furthermore, every induced bipartite subgraph $F \subseteq H_n$ has minimum degree at most $\delta(F) \le \lfloor n/n'\rfloor \cdot  \biC \log n' \le 2\alpha^{-1}\biC \cdot d^2/n$. 
By blowing up each vertex of~$H_n$ into an independent set of suitable sizes between one and six, 
we then obtain an $n$-vertex triangle-free graph~$G_{n,d}$ that establishes~$g(n,d) = O(d^2/n)$. 
\end{proof}

\subsubsection{Construction of pseudo-random $r$-AP free sets of integers}\label{sec:techn:vdW}
%
To prove the lower bound on the van der Waerden number~$W(r,k)$ claimed by \refT{thm:van}, 
our strategy is to construct a large subset~$I \subseteq [n]$ of the integers that is $r$-AP free and pseudo-random, 
where pseudo-randomness will intuitively ensure that $[n] \setminus I$ is $k$-AP free for fairly large~$k=k(n)$. 
For technical reasons, it will be convenient to work with the field $\mathbb{Z}/N\mathbb{Z}$ for a prime number~$N$, 
where a set of numbers $\{a_1,\dots, a_r\}\subseteq~\mathbb{Z}/N\mathbb{Z}$ is formally called an $r$-term arithmetic progression ($r$-AP) in~$\mathbb{Z}/N\mathbb{Z}$ if~$|\{a_1,\dots,a_r\}|=r$ and $a_i\equiv_N a_1+(i-1)d$ for some~$d \not\equiv_N 0$. 
Following the conceptual punchline of this paper, we shall construct the desired pseudo-random $r$-AP free subset $I\subseteq \mathbb{Z}/N\mathbb{Z}$ using the so-called random greedy $r$-AP free process, 
which is a randomized greedy algorithm that step-by-step adds more random numbers to~$I$ without creating an $r$-AP, see~\refS{sec:van} for the full details. 
This algorithmic approach to the probabilistic method is key for obtaining our improved lower bound on~$W(r,k)$ via the following result, 
since earlier approaches based on random subsets the integers 
were only able to prove \refT{thm:AP} with the weaker parameter choice~$k = \Theta(N^{1/(r-1)}\log N)$, see~\cite{li2010lower}. 
Note that we recover these earlier results, by observing that the set~$I \subseteq \mathbb{Z}/N\mathbb{Z}$ from \refT{thm:AP} 
also satisfies~$|I\cap K|\ge 1$ for all $k$-APs~$K$ in~$\mathbb{Z}/N\mathbb{Z}$ of size~$k \ge k_N$.  
\begin{theorem}\label{thm:AP}
For any fixed $r\ge 3$, there are constants~$C,N_0>0$ such that the following holds for any prime number $N \ge N_0$. 
There exists a set~$I \subseteq \mathbb{Z}/N\mathbb{Z}$ which (i)~is $r$-AP free in~$\mathbb{Z}/N\mathbb{Z}$ 
and (ii)~satisfies~$|I\cap K|\ge 1$ for all $k$-APs~$K$ in~$\mathbb{Z}/N\mathbb{Z}$ of size~$k=k_N := \ceil{C(N/\log N)^{1/(r-1)}\log N}$. 
\end{theorem}
\begin{proof}[Proof of \refT{thm:van} based on~\refT{thm:AP}]
Assuming that~$k$ is sufficiently large (as we may),  
we pick the largest prime number~$N \ge \max\{2,N_0\}$ satisfying~$k \ge k_N =\ceil{C(N/\log N)^{1/(r-1)}\log N}$. 
Using Bertrand's postulate it follows that~${n :=N-1}= {\Theta(k^{r-1}/(\log k)^{r-2})}$. 
For~${I \subseteq \mathbb{Z}/N\mathbb{Z}}$ as given by \refT{thm:AP}, 
we color~$I \cap [n]$~red and $[n] \setminus I$~blue. 
Properties~\mbox{(i)--(ii)} of \refT{thm:AP} and~$k \ge k_N $
ensure that this coloring contains no red \mbox{$r$-APs} or blue~\mbox{$k$-APs} in~$[n]$, 
since any AP in~$[n]$ corresponds to an AP in~$\mathbb{Z}/N\mathbb{Z}$ 
(and any blue~\mbox{$k$-AP} contains all numbers of at least one blue~\mbox{$k_N$-AP}). 
It follows that~${W(r,k) > n} = \Theta(k^{r-1}/(\log k)^{r-2})$.  
\end{proof}
We defer the proof of \refT{thm:AP} to \refS{sec:van}: 
it is based on the differential equation method and results of Bohman and Bennett~\cite{bennett2016note} for the so-called random greedy independent set algorithm. 
Noteworthily, in our analysis we need to ensure that all of the polynomially many $k$-APs are `hit' by the set~$I$ produced by the $r$-AP process. 
This is in great contrast to the analysis of the $H$-free process arising in graph Ramsey theory, 
where one typically needs to ensure that an exponential number of substructures are hit~\cite{bohman2009triangle,bohman2010early,pontiveros2013triangle,bohman2013dynamic,Warnke2014,Warnke2014when,Picollelli2014final}.

\section{Semi-random triangle-free process}\label{sec:gnd}
In this section we prove \refT{thm:semi} by showing that a semi-random variant of the so-called triangle-free process typically finds a triangle-free graph $G_n \subseteq K_n$ with the desired properties.
Intuitively, this process starts with an empty graph, and then iteratively adds a large number of carefully chosen edges (instead of just adding a single edge as in the original triangle-free process) such that the resulting graph stays triangle-free.

\subsection{More details and heuristics}\label{subsec:basicssemi-random}
The formal details of the semi-random triangle-free process given in~\mbox{\cite[Section~2]{guo2020packing}} are rather involved, 
so here we shall only introduce those aspects that are important for the upcoming arguments of this paper 
(keeping the notation from~\cite{guo2020packing} to minimize notational differences). 
%
The semi-random process starts with 
\begin{equation}\label{eq:def:initial}
E_0=T_0:=\emptyset \quad \text{ and } \quad O_0:=E(K_n)  ,
\end{equation}
and the rough plan is to step-by-step build up 
a `random' set of edges~$E_i \subseteq E(K_n)$, 
a triangle-free edge subset~$T_i \subseteq E_i$, 
and a set of `open' edges~$O_i \subseteq E(K_n) \setminus E_i$, each of which can still be added to $E_i$ without creating triangles. 
More precisely, in step $i+1 \ge 1$ of the semi-random triangle-free process 
we sample a random edge subset~$\Gamma_{i+1} \subseteq O_i$, 
where each edge~$e \in O_i$ is included independently with probability 
\begin{equation}\label{eq:def:probp}
p:=\sigma/\sqrt{n} \quad \text{ for } \quad  \sigma:=(\log n)^{-2}, 
\end{equation}
and update the random set of edges by~setting 
\begin{equation}\label{def:Ei1}
E_{i+1}:=E_i\cup\Gamma_{i+1}. 
\end{equation}
To determine the new triangle-free edge subset~$T_{i+1} \subseteq T_i \cup\Gamma_{i+1}$, 
the idea is to delete a suitable set~$D_{i+1} \subseteq \Gamma_{i+1}$  of edges from~$\Gamma_{i+1}$ with $|\Gamma_{i+1} \setminus D_{i+1}| \approx |\Gamma_{i+1}|$, 
such~that 
\begin{equation}\label{def:Ti1}
T_{i+1} := T_i \cup \bigl(\Gamma_{i+1} \setminus D_{i+1}\bigr) 
\end{equation}
remains triangle-free, see~\cite[(13)--(14) in~Section~2.1]{guo2020packing} for the precise definition of~$D_{i+1}$ 
(this construction intuitively works since only few new triangles are created in~$E_i\cup\Gamma_{i+1}$ due to the fact that $\Gamma_{i+1}$ is fairly small).  
To determine the new open edge set~$O_{i+1} \subseteq O_i \setminus \Gamma_{i+1}$, 
we certainly have to remove the set~$C'_{i+1}$ of `newly closed' edges, which simply contains all edges~$e \in O_i$ that form a triangle with some two edges of~$E_{i+1}=E_i\cup\Gamma_{i+1}$.
For technical reason we also remove an extra random edge subset~$S_{i+1} \subseteq O_i$ and set 
\begin{equation}\label{def:Oi:subset}
O_{i+1} := O_i \setminus \bigl(\Gamma_{i+1} \cup C'_{i+1} \cup S_{i+1}\bigr) ,
\end{equation}
see~\cite[(15)--(20) in~Section~2.1]{guo2020packing} for the precise definition of~$C'_{i+1} \cup S_{i+1}$ 
(the removal of extra edges is a technical twist that intuitively makes it easier to prove certain concentration statements).

Stopping this iterative 
construction after~$I \approx n^{\beta}$ steps, 
the pseudo-random intuition from~\mbox{\cite[Section~2]{guo2020packing}} 
suggests that, with respect to various edge statistics, the resulting $n$-vertex triangle-free graph
\begin{equation}\label{eq:def:I}
H:=\bigl([n], T_I\bigr) \quad \text{ with } \quad I:=\big\lceil n^\beta\big\rceil 
\end{equation}
heuristically resembles a binomial random graph~$G(n,\rho)$ with edge~probability
\begin{equation}\label{eq:def:rho}
\rho := \sqrt{\beta (\log n)/n} ,
\end{equation}
but with the notable exception that it by construction contains no triangles  
(such a random graph would typically contain many triangles). 
This heuristic makes it plausible that~${G_n= H}$ satisfies the degree properties claimed by~\refT{thm:semi}, 
since routine arguments show that the random graph~$G(n,\rho)$ typically has these degree properties. 
To keep the modifications of~\cite{guo2020packing} minimal, 
we shall in fact find an induced subgraph~${G_n \subseteq H}$ with the desired degree properties 
(this extra step is convenient but not necessary, see~\refR{rem:degree}).

\subsection{Setup and proof of \refT{thm:semi}}\label{subsec:analysisoftrianglefree}
We now turn to the technical details of our proof of \refT{thm:semi}, which extends~\cite[Sections~2--3]{guo2020packing}. 
Here our setup is guided by the pseudo-random heuristic discussed in~\cite[Section~2.2]{guo2020packing}, 
which loosely suggests that 
\begin{equation}\label{eq:intuition}
\cPr(e\in E_i)\approx \pi_i/\sqrt{n}
 \quad\text{ and }\quad 
\cPr(e\in O_i )\approx q_i,
\end{equation}
where the parameters~$\pi_i$ and~$q_i$ defined in~\cite[Section~2.3]{guo2020packing} 
satisfy the technical properties 
\begin{equation}
\label{eq:relationpiI}
\pi_i:=\sigma +\sum_{0 \le j < i}\sigma q_j,  
\quad 
0 < q_i \le 1 = q_0
\quad \text{ and } \quad 
\pi_I/\sqrt{n} = (1+o(1)) \rho,  
\end{equation}
see~\cite[Section~2.3 and Lemma~17]{guo2020packing} for the full details,  
which formally justify that we may indeed use the parameters~$\pi_i$ and~$q_i$ 
and their properties~\eqref{eq:relationpiI} in this paper 
(in contrast to the heuristic approximations~\eqref{eq:intuition}, which we of course may not use in our proofs). 
In particular, to get a handle on the number of edges between large sets of vertices, 
consistent with~\eqref{eq:intuition}--\eqref{eq:relationpiI} we introduce the pseudo-random~events 
\begin{align}
\label{eq:eventTI}
	\cT^*_I &:= \Big\{|T_I(A,B)| \ge (1-\delta)|A||B| \rho \text{ for all disjoint $A, B\subseteq [n]$ with $|A|=|B|=s $}   \Big\} ,\\
	\label{eq:evencTI}
	\cT^+_I & := \Big\{|T_I(A,B)|\le (1+\delta) 2s |B| \rho \text{ for all disjoint $A, B\subseteq [n]$ with $1 \le |A|=|B| \le 2s $} \Big\},  
\end{align}
where we write $S(A,B):=\{ab\in S: a\in A, b\in B  \}$ for the set of edges from $S$ that go between $A$ and $B$,
and use the carefully chosen (see~\cite[Section~2.3 and Theorem~9]{guo2020packing}) size parameter
\begin{equation}\label{eq:def:sS}
s:= \bigl\lceil D (\log n)/\rho \bigr\rceil \quad \text{ with } \quad D:= 108/\delta^2 \quad \text{and} \quad \delta := 1/10 .
\end{equation}
To eventually get a handle on the maximum degree, 
we also introduce the auxiliary~event 
\begin{align}
	\label{eq:eventNi}
	\mathcal{N}_{\le I} &:= \Big\{|N_{\Gamma_{i}}(v)|\le 2\sigma q_{i-1}\sqrt{n} \text{ for all $v\in [n]$ and $0 < i \le I$}  \Big\} ,
\end{align}
writing~$N_S(v):=\{ w\in [n]:vw\in S \}$ for the set of neighbors of~$v$ in a given edge set~$S$.

Results of Guo and Warnke, see~\cite[Theorem~9]{guo2020packing}, imply that
\begin{equation}\label{eq:probTX}
	\P\bigpar{\cT^*_I\cap \mathcal{N}_{\le I}} \ge 1-n^{-\omega(1)} . 
\end{equation}
As we shall show next, \refT{thm:semi} then follows from the claim
\begin{equation}\label{eq:probDTX:goal}
	\cPr\bigpar{\cT^+_I}  \ge 1-o(1) ,
\end{equation}
whose stochastic domination based proof we defer to~\refS{sec:semirandom}. 
\begin{proof}[Proof of \refT{thm:semi} assuming inequality~\eqref{eq:probDTX:goal}]
Combining~\eqref{eq:probTX}--\eqref{eq:probDTX:goal} we infer that~$\P(\cT^*_I \cap \mathcal{N}_{\le I} \cap \cT^+_I) > 0$ for all sufficiently large~$n$, 
so by the probabilistic method we may henceforth fix a graph~${H=([n], T_I)}$ for which the event~${\cT^*_I \cap \mathcal{N}_{\le I} \cap \cT^+_I}$ holds.
We then construct the induced triangle-free subgraph~${G_n \subseteq H}$ by iteratively deleting vertices of degree at most $\delta/4 \cdot n\rho$, 
and now verify that it has the claimed properties, starting with the degree bound~\eqref{eq:maxmindegree}.
Noting $e_H(A,B)=|T_I(A,B)|$, the event~$\cT^*_I$ implies, via a double-counting argument for $e(H)$ analogous to the proof of~\cite[Theorem~5]{guo2020packing}, that the number of edges of~$H$ is at~least
\begin{gather}\label{eq:H:edges}
e(H) = \frac{\sum_{A \subseteq [n]: |A|=s} \sum_{B \subseteq [n] \setminus A: |B|=s}|T_I(A,B)|}{2\binom{n-2}{s-1}\binom{n-s-1}{s-1}}
\ge \frac{\binom{n}{s}\binom{n-s}{s} \cdot (1-\delta)s^2\rho}{2\binom{n-2}{s-1}\binom{n-s-1}{s-1}} = (1-\delta)\tbinom{n}{2}\rho. 
\end{gather}
Furthermore, by the recursive definition~\eqref{def:Ti1} of the edge set~${T_I \subseteq \bigcup_{0 \le i < I}\Gamma_{i+1}}$, 
using the properties~\eqref{eq:relationpiI} of~$\pi_i$ we infer, for all sufficiently large~$n \ge n_0(\beta)$, that the event~$\mathcal{N}_{\le I}$ implies the maximum degree bound
\begin{equation*}  
\Delta(H) = \max_{v\in [n]} |N_{T_I}(v)| \le \max_{v\in [n]} \sum_{0 \le i < I}|N_{\Gamma_{i+1}}(v)|\le \sum_{0 \le i < I}2\sigma q_i\sqrt{n} \le 2\pi_I\sqrt{n} \le (2+\delta) \rho n .
\end{equation*} 
By construction of~$G_n \subseteq H$, 
using $\delta=1/10$ we thus infer, for all sufficiently large~$n \ge n_0(\beta)$, that 
\[
v(G_n) \ge \frac{2 e(G_n)}{\Delta(G_n)} \ge \frac{2 \bigl[e(H) - n \cdot \delta/4 \cdot n\rho\bigr]}{\Delta(H)} \ge \frac{2(1-2\delta)\binom{n}{2}\rho}{(2+\delta)n\rho} > \frac{n}{3} ,
\]
and so the claimed degree bound~\eqref{eq:maxmindegree} follows with~$\lbc := \delta/4$ and~$\ubC := 2+\delta$. 

Next, suppose that~$F \subseteq G_n$ is an induced bipartite subgraph with two parts~$A$ and~$B$, where we may assume that~$|A| \ge |B| \ge 1$. 
Since~${F \subseteq G_n}$ and $G_n \subseteq H$ are both induced subgraphs, we~have
\[
e_F(A,B)=e_{G_n}(A,B)=e_H(A,B)=|T_I(A,B)|.
\]
Furthermore, since~$A$ and~$B$ are both independent sets in~$F$, we have~${|B|\le |A| \le \alpha(F)} \le {\alpha(H) \le 2 \cdot s}$, 
where the last inequality holds because the event~$\cT_I^*$ implies 
that in~$H$ there is at least one edge between any two disjoint $s$-vertex sets. 
Using a double counting argument similar to~\eqref{eq:H:edges}, 
the event~$\cT^+_I$ then implies~that 
\[ |T_I(A,B)| = \frac{\sum_{A'\subseteq A:|A'|=|B|}|T_I(A',B)|}{\binom{|A|-1}{|B|-1}}\le \frac{\binom{|A|}{|B|} \cdot (1+\delta) 2s |B|\rho}{\binom{|A|-1}{|B|-1}}=(1+\delta) 2s |A|\rho.\]
The definitions~\eqref{eq:def:sS} of~${s \approx D (\log n)/\rho}$ and~$\delta=1/10$ give~${(1+\delta) 2s\rho \le 3D \log n}$ for sufficiently large~$n$. 
By averaging it follows that~${\delta(F) \le e_F(A,B)/|A| \le 3D \log n}$, 
completing the proof with~$\biC:=3D$. 
\end{proof}
\begin{remark}\label{rem:degree}
For any~$\delta>0$ one can in fact show that the minimum and maximum degree of the $n$-vertex graph~${H=([n], T_I)}$  
satisfy ${(1-\delta) n\rho \le \delta(H) \le \Delta (H) \le (1+\delta) n\rho}$ with high probability
(by adapting~\cite[Sections~3.1--3.5]{guo2020packing}), 
which would allow us to directly use~$G_n=H$ in the above proof of~\refT{thm:semi}. 
However, the coarser bounds used above suffice for our purposes, 
and require less technical modifications of~\cite{guo2020packing}. 
\end{remark}

\subsection{Pseudo-randomness: deferred proof of inequality~\eqref{eq:probDTX:goal}}\label{sec:semirandom} 
This subsection is devoted to the deferred proof of inequality~\eqref{eq:probDTX:goal}, i.e., $\cPr(\cT^+_I)  \ge 1-o(1)$.  
To this end we shall adapt the strategy from~\cite[Sections 3.4--3.5]{guo2020packing} to our setting, 
i.e., use estimates on the number of open edges~$|O_i(A,B)|$ to eventually get a handle on the total number of added edges~$|T_I(A,B)|$. 

Turning to the details, let~$\cS$ denote the set of all pairs of vertex disjoint $A,B \subseteq [n]$ with $1 \le |A|=|B| \le 2 s$.
To keep the changes to~\cite{guo2020packing} minimal, for each pair~$(A,B) \in \cS$ we enlarge~$A$ to~$A^+$ by adding the lexicographic first~$2s-|A|$ vertices from~$[n] \setminus (A \cup B)$. 
Note that the vertex set~$A^+$ is determined by~$A$. 
Consistent with the heuristic approximations~\eqref{eq:intuition}, we then introduce the `open edges' related pseudo-random~events
\begin{equation}\label{def:Qpi}
\ctQ_i^+:=\Bigl\{ \bigabs{O_i(A^+, B)} \le q_i|A^+||B| \text{ for all $(A, B)\in \cS$} \Bigr\} 
\quad \text{ and } \quad 
\ctQ^+_{\le I}:=\bigcap_{0 \le i \le I}\tilde{\mathcal{Q}}^+_i.
\end{equation}
Note that~$1 \le |B| \le |A^+|=2s$ for all pairs~$(A,B) \in \cS$. 
Furthermore, there are at most~$n^{2j}$ pairs~$(A,B) \in \cS$ with~$|B|=j$. 
With these two key properties in mind, the proof of~\cite[Lemma~24]{guo2020packing} carries over to the pairs~$(A^+,B)$ virtually unchanged 
(that proof merely exploits that~$|A|$ is large, and only uses~$|A|=|B|$ to control the final union bound estimate over all pairs~$(A,B)$ of vertex subsets), 
giving 
\begin{equation*}
\max_{0 \le i < I}\cPr\bigl(\neg \ctQ^+_{i+1} \: \mid \: \ctQ^+_{\le i} \cap \mathfrak{X}_{\le i} \bigr) \: \le \: \sum_{(A,B) \in \cS} n^{-\omega(|B|)} \le \sum_{1 \le j \le 2s}n^{2j-\omega(j)} \le n^{-\omega(1)},
\end{equation*}
where~$\mathfrak{X}_{\le i}$ is a `good' event determined by~$(O_j,E_j,T_j,\Gamma_j,S_j)_{0 \le j \le i}$ that is formally defined in~\cite[Section~2.4]{guo2020packing};
here we shall only use that the event~$\mathfrak{X}_{\le i+1}$ implies~$\mathfrak{X}_{\le i}$, and that~$\cPr(\neg \mathfrak{X}_{\le I}) \le n^{-\omega(1)}$ by~\cite[Theorem~9]{guo2020packing}.
In view of~$q_0=1$, see~\eqref{eq:relationpiI}, it is straightforward to check that the event~$\ctQ_0^+ = \ctQ_{\le 0}^+$ always holds. 
Since the event~$\mathfrak{X}_{\le I}$ implies~$\mathfrak{X}_{\le i}$ for all~$0 \le i < I$, 
using~$\ctQ^+_{\le i+1} = \ctQ^+_{i+1} \cap \ctQ^+_{\le i}$ and~$I \approx n^{\beta}$ it follows~that 
\begin{equation}\label{eq:probDTX:Q}
\begin{split}
    \cPr\bigl(\neg \ctQ^+_{\le I}\bigr) \: &\le \: \cPr\bigl(\neg \mathfrak{X}_{\le I}\bigr)+\cPr(\neg\ctQ^{+}_{i+1} \cap \ctQ^{+}_{\le i} \cap \mathfrak{X}_{\le i} \text{ for some $0\le i<I$} ) \\
&\le\: \cPr\bigl(\neg \mathfrak{X}_{\le I}\bigr) + \sum_{0 \le i <I} \cPr\bigl(\neg \ctQ^+_{i+1} \: \mid \: \ctQ^+_{\le i} \cap \mathfrak{X}_{\le i} \bigr) \: \le \: (I+1) \cdot n^{-\omega(1)} \le n^{-\omega(1)}.
\end{split}
\end{equation}

Turning to the total number of added edges~$|T_I(A,B)|$ for $(A,B)\in\cS$, 
using~$A \subseteq A^+$ and~$T_I \subseteq E_I$  
together with the recursive definition~\eqref{def:Ei1} of the edge set~$E_I = \bigcup_{0 \le i < I}\Gamma_{i+1}$,  
it follows~that 
\begin{equation}\label{eq:probDTX:edgesbound}
|T_I(A,B)| 
\le \bigabs{E_I(A^+,B)} 
= \sum_{0 \le i < I}\bigabs{O_i(A^+,B) \cap \Gamma_{i+1}} .
\end{equation}
Recall that the event~$\ctQ_i^+$ implies $|O_i(A^+,B)| \le q_i|A^+||B|$, 
and that~$\Gamma_{i+1} \subseteq O_i$ is the random subset where each edge $e \in O_i$ is included independently with probability~$p$. 
Combining these properties, by mimicking the stochastic domination arguments from the proof of~\cite[Claim~30]{guo2020packing} it then follows that 
\[
\cPr\bigpar{|E_I(A^+,B)| \ge t \text{  and  } \ctQ^+_{\le I}} \le \cPr\bigl(Z^+ \ge t\bigr) \quad \text{ with } \quad Z^+ \overset{\mathrm{d}}{=} \Bin\Bigpar{\sum_{0\le i< I}\bigfloor{q_i|A^+||B|}, \: p}.
\]
Using~$p=\sigma/\sqrt{n}$, the properties~\eqref{eq:relationpiI} of~$\pi_i$ and~$|A^+|=2s$, similar to~\cite[Section~3.5]{guo2020packing} we have 
\[
\mu^+ := \E Z^+ \sim \sum_{0 \le i < I}\sigma q_i/\sqrt{n} \cdot |A^+||B| = (\pi_I-\sigma)/\sqrt{n} \cdot |A^+||B| \sim 2 s|B|\rho  .
\] 
Using the definitions~\eqref{eq:def:sS} of the parameter~${s \approx D (\log n)/\rho}$ and the constant~${D = 108/\delta^2}$,
for sufficiently large~$n$ it follows that~$(1+\delta) 2s |B|\rho \ge (1+\delta/2)\mu^+$ and~$\delta^2 \mu^+/12 > 12 |B| \log n$, say. 
Similar to~\mbox{\cite[(97)--(98)]{guo2020packing}}, standard Chernoff bounds such as~\cite[Theorem~2.1]{JLR} thus routinely give
\[
\cPr\bigl(|E_I(A^+,B)| \ge (1+\delta) 2s |B|\rho \text{ and } \ctQ^+_{\le I}\bigr) \le \cPr\bigl(Z^+ \ge (1+\delta/2)\mu^+\bigr) 
\le \exp\bigl(-\delta^2\mu^+/12\bigr) 
\le n^{-12|B|}.
\]
Recalling that the vertex set~$A^+$ is determined by~$A$, and that there are at most most~$n^{2j}$ pairs~$(A,B) \in \cS$ with~$|B|=j$, 
in view of inequality~\eqref{eq:probDTX:edgesbound} it then follows via a standard union bound argument that 
\[
\cPr\bigl(\neg \cT^+_I \cap \ctQ^+_{\le I}\bigr) \le  \sum_{(A,B) \in \cS} n^{-12|B|} \le \sum_{1 \le j \le 2s}n^{2j-12j} = o\bigpar{n^{-9}},
\]
which together with~\eqref{eq:probDTX:Q} implies~$\cPr(\cT^+_I)  \ge 1-o(1)$. 
This completes the proof of inequality~\eqref{eq:probDTX:goal} and thus \refT{thm:semi}, as discussed. 
\noproof

\section{Random greedy $r$-AP free process}\label{sec:van}
In this section we prove \refT{thm:AP} by showing that the random greedy $r$-AP free process typically finds an $r$-AP free subset~$I \subseteq \mathbb{Z}/N\mathbb{Z}$ with the desired properties.
Intuitively, this process starts with an empty set~$I=\emptyset$, and then iteratively adds new random numbers from~$\mathbb{Z}/N\mathbb{Z}$ such that the resulting set~$I$ stays $r$-AP~free.
More formally, fixing~$r \ge 3$, the random greedy $r$-AP free process starts with 
\begin{equation}
I(0):=\emptyset \quad\text{ and }\quad \cV(0):=\mathbb{Z}/N\mathbb{Z} .
\end{equation}
Here $I(i)$ denotes the growing $r$-AP free set found after $i$ steps, 
and $\cV(i)$ denotes the set of `available' numbers in $\mathbb{Z}/N\mathbb{Z} \setminus I(i)$, i.e., that can be added to~$I(i)$ without creating an $r$-AP.
In step~$i+1 \ge 1$ of the random greedy $r$-AP free process, 
we then choose $x_{i+1}\in \cV(i)$ uniformly at random and update the $r$-AP free set and available set via  
\begin{align}
I(i+1) & :=I(i)\cup\{ x_{i+1} \},\\
\cV(i+1) &:=\cV(i)\setminus \bigpar{\{\cv_{i+1}\}\cup \NI_{\cv_{i+1}}(i)},
\end{align}
where we write $\NI_{\cv_{i+1}}(i)$ for the set of numbers that become `unavailable' when $\cv_{i+1}$ is added, i.e., 
\begin{equation} 
\NI_\cv(i):=\big\{\cu\in \cV(i) \setminus \{\cv\} \: : \: \text{there is $A\in \cA_{N,r}$ such that $\cv,\cu\in A$ and $A\setminus\{\cv,\cu\}\subseteq \cI(i)$}  \Big\} ,
\end{equation} 
in which~$\cA_{N,\ell}$ is a shorthand for the collection of all $\ell$-APs in~$\mathbb{Z}/N\mathbb{Z}$.

\subsection{Proof strategy}
In this subsection we discuss our proof strategy for \refT{thm:AP}. 
To this end, let us first record the basic observation that each number~$x\in\mathbb{Z}/N\mathbb{Z}$ is contained in~exactly 
\begin{equation}\label{eq:D}
D:=r|\cA_{N,r}|/N = \Theta(N)
\end{equation} 
many {$r$-APs}~${A\in\cA_{N,r}}$. 
Our strategy is then to analyze the random greedy $r$-AP free process for
\begin{equation}\label{eq:def:m}
	m:=\cxi \cdot N\cmD (\log N)^{\frac{1}{r-1}} 
\end{equation}
steps, and show that the $r$-AP free set $I:=I(m) \subseteq \mathbb{Z}/N\mathbb{Z}$ typically satisfies~$I\cap K\neq\emptyset$ for all $k$-APs $K\in\cA_{N,k}$ of size 
\begin{equation}\label{eq:def:k}
k:=9\cxi^{-1} \cdot (D/\log N)^{1/(r-1)}\log N = \Theta\bigpar{(N/\log N)^{1/(r-1)}\log N},
\end{equation}
deferring the choices of the sufficiently small constants~$0 < {\cxi,\delta < 1/(2r)}$. 
As usual, we are henceforth treating both~$m$ and~$k$ as integers (since rounding has an asymptotically negligible effect on our~arguments).

The outlined proof strategy is consistent with the pseudo-random heuristic that $I=I(m)$ resembles a random $m$-element subset of $\mathbb{Z}/N\mathbb{Z}$. 
Indeed, noting $k\tm=9N\log N$, this heuristic suggests that 
\[
\cPr\bigpar{I\cap K = \emptyset} 
\approx \frac{\binom{N-k}{m}}{\binom{N}{m}} 
= \prod_{0 \le j < k}\Bigpar{1-\frac{m}{N-j}} 
\le \exp\Bigpar{-\frac{k \tm}{N}} \ll N^{-2},
\]
which is small enough to employ a union bound argument over the at most $N^2$ many~$k$-APs~$K\in\cA_{N,k}$. 
In~\eqref{eq:IcapKnonempty} below and \refS{sec:dynamic} we will essentially make this heuristic reasoning rigorous, albeit in a slightly roundabout~way
(via several pseudo-random events and the differential equation method).

\subsection{Setup and proof of~\refT{thm:AP}}
We now turn to the technical details of our proof of \refT{thm:AP}, which require some setup. 
In order to relate the discrete steps of the process to continuous trajectories, 
we introduce the convenient scaling 
\begin{equation}
t_i :=i/\cM \quad \text{ with } \quad \cM:=N \cmD .
\end{equation}
To get a handle on all~$k$-APs $K\in \cA_{N,k}$, we denote the number of available numbers in $K$ by 
\begin{equation}
\cQ_K(i):=\cV(i)\cap K .
\end{equation}
Henceforth using the shorthand~$X=(a\pm b)x$ for~$X\in [(a-b)x,\: (a+b)x]$ to avoid clutter (as usual),
we then define~$\mathcal{K}_{\le j}$ as the pseudo-random event that
for all $0 \le i \le j$ we have 
\begin{equation}\label{eq:QKibound}
 |\cQ_K(i)|=  \big(1\pm \err(t_i)\big)kq(t_i) \qquad \text{for all $K\in \cA_{N,k}$,}
\end{equation}
and similarly define $\mathcal{S}_{\le j}$ as the pseudo-random event that for all $0 \le i \le j$ we have
\begin{equation}\label{eq:SiBound}
|\cV(i)| = \big(1\pm D^{-\delta}\big)Nq(t_i)
\quad \text{ and } \quad
\max_{\cv\in \cV(i)}\bigabs{|\NI_\cv(i)| - s_2(t_i)} \le D^{\frac{1}{r-1}-\delta},
\end{equation}
using the deterministic functions
\begin{equation}\label{eq:APPrDefs}
q(t) :=e^{-t^{r-1}}, \quad  s_2(t) := (r-1)\cpD t^{r-2}q(t) \quad \text{ and } \quad  \err(t) := e^{5 (t+t^{r-1})} \cdot D^{-\delta}.
\end{equation}
Note that, by choosing ${\cxi=\cxi(r,\delta)>0}$ small enough   
compared to~${r,\delta>0}$, 
we may assume that for all \mbox{steps}~${0 \le i \le m}$ we have ${0\le t_i \le t_m= m/ \cM=\cxi(\log N)^{\frac{1}{r-1}}}$ as well as  
\begin{equation}\label{eq:ratios}
0< D^{-\delta}\le  \err(t)=o(1) \quad \text{ and } \quad  0 \le t \le D^{o(1)} \quad\text{ for }\quad 0\le t\le t_m.
\end{equation}

Results of Bohman and Bennett (which require~$N$ to be prime), see~\cite[Section~4]{bennett2016note}, imply that for sufficiently\footnote{For example, the explicit choices~$\delta=1/(40r^2)$, $\cxi=\delta/500$ satisfy all constraints of this paper and~\cite{bennett2016note}.} small~${\cxi,\delta>0}$ we have  
\begin{equation}\label{eq:probSm}
\cPr(\neg\mathcal{S}_{\le \tm}) \le \exp\bigpar{-N^{\Omega(1)}}.
\end{equation}
Using estimates~\eqref{eq:SiBound} and~\eqref{eq:ratios}, 
we see that the event~$\mathcal{S}_{\le \tm}$ implies ${\min_{0 \le i \le m}|\cV(i)| = |\cV(m)| \ge Nq(t_m)/2 > 0}$ for all sufficiently large~$n$, 
which in turn ensures that the random greedy $r$-AP free process does not terminate before step~$m$ 
(since the process is always able to select a new number in each of the first~$m$ steps). 
As we shall show next, \refT{thm:AP} then follows easily from the claim 
\begin{equation}\label{eq:sumGm:bound}
\cPr(\neg \cGo_{\le \tm}) = o(1) \quad \text{ for } \quad \cGo_{\le i}:=\mathcal{S}_{\le i} \cap \mathcal{K}_{\le i} ,
\end{equation}
whose differential equation method based proof we defer to \refS{sec:dynamic}. 
\begin{proof}[Proof of~\refT{thm:AP} assuming inequality~\eqref{eq:sumGm:bound}]
For any $k$-AP~$K \in \cA_{N,k}$ in~$\mathbb{Z}/N\mathbb{Z}$,
whenever the event~$\cGo_{\le i}$ holds, 
by combining the concentration bounds~\eqref{eq:QKibound}--\eqref{eq:SiBound}   
with the error estimate~\eqref{eq:ratios} we infer that 
\begin{equation*}
\cPr(\cv_{i+1} \not\in \cV(i)\cap K\mid\mathcal{F}_i)=	1-\frac{|\cQ_K(i)|}{|\cV(i)|} \le 1-\frac{\frac{1}{2}kq(t_i)}{2Nq(t_i)}  =1-\frac{k}{4N},
\end{equation*}
where~$\mathcal{F}_i$ denotes the natural filtration associated with the process after~$i$ steps 
(which intuitively keeps track of the `history' of the process, i.e., all the information available up to and including step~$i$).
Since the event $\cGo_{\le \tm}$ implies the event~$\cGo_{\le i}$ for all $0 \le i \le \tm$, 
using~$km = 9 N \log N$ it routinely follows that 
\begin{equation}\label{eq:IcapKnonempty}
\begin{split}
\cPr\bigpar{I(m) \cap K =\emptyset \ \text{ and } \ \cGo_{\le \tm}}
\le \prod_{0\le i\le \tm-1}\Big(1-\frac{k}{4N}  \Big)\le \exp\Bigpar{-\frac{k \tm}{4N}} \ll N^{-2}.
\end{split}
\end{equation}
Taking a union bound over the at most $N^2$ many $k$-APs~$K \in \cA_{N,k}$ in~$\mathbb{Z}/N\mathbb{Z}$
then completes the proof of~\refT{thm:AP} with~$I:=I(m)$, 
since $\cPr(\neg\cGo_{\le\tm})=o(1)$ by the assumed inequality~\eqref{eq:sumGm:bound}. 
\end{proof}

\subsection{Dynamic concentration: deferred proof of inequality~\eqref{eq:sumGm:bound}}\label{sec:dynamic}
This subsection is devoted to the deferred proof of inequality~\eqref{eq:sumGm:bound}, i.e.,~$\cPr(\neg \cGo_{\le \tm}) = o(1)$, 
which in view of the probability estimate~\eqref{eq:probSm} 
and the definition of the event~$\mathcal{K}_{\le i}$ 
requires us to establish the dynamic concentration estimate~\eqref{eq:QKibound} for~$|\cQ_K(i)|$.
To this end, following the differential equation method approach to dynamic concentration~\cite{wormald1995,bohman2009triangle,warnke2019wormald}, 
for all~$k$-APs $K\in \cA_{N,k}$, steps~$0 \le i \le m$ and sign-parameters~$\sigma \in \{+,-\}$ 
 we introduce the auxiliary random variables
\begin{equation}\label{eq:def:X_ipm}
X_K^{\sigma}(i):=\sigma [|\cQ_K(i)|-kq(t_i)]-kq(t_i)\err(t_i) .
\end{equation}
The point is that the desired estimate~\eqref{eq:QKibound} follows when the inequalities $X_K^{+}(i)\le 0$ and $X_K^{-}(i)\le 0$ both hold. 
In the following we shall use supermartingale arguments to establish these inequalities, 
by analyzing the (expected and worst-case) one-step changes of~$X_K^{\sigma}(i)$ and~$|\cQ_K(i)|$.

\subsubsection{Expected one-step changes}\label{subsubsec:super}
We start by estimating the expected value of the one-step changes $\Delta \cQ_{K}(i):=|\cQ_K(i+1)|-|\cQ_K(i)|$ of the number of available numbers in any~$k$-AP~$K \in \cA_{N,k}$, 
assuming that $0 \le i < m$ and $\cGo_{\le i}$ hold.
Note that~$|\cQ_K(i)|$ is monotone decreasing. 
Furthermore, a number~$\cv \in \cQ_K(i)$ is removed from the set of available numbers 
if the process chooses a number~$\cv_{i+1}$ from~$\NI_{\cv}(i) \cup \{\cv\}$. 
Since~$\cv_{i+1} \in \cV(i)$ is chosen uniformly at random, 
using the estimates~\eqref{eq:QKibound}--\eqref{eq:SiBound} implied by~$\cGo_{\le i}$ it follows that 
\begin{equation}
\label{eq:expectedQKbound:0}
\E(\Delta \cQ_{K}(i)\mid\mathcal{F}_i)
= -\sum_{\cv\in \cQ_K(i)}\frac{|\NI_{\cv}(i)| \pm 1}{|\cV(i)|} = \frac{-[1\pm \err(t_i)]kq(t_i)\cdot\bigl[s_2(t_i)\pm 2D^{\frac{1}{r-1}-\delta}\bigr]}{[1\pm D^{-\delta}]Nq(t_i)} .
\end{equation}
Recalling that~$0<D^{-\delta}\le \err(t_i) = o(1)$ by~\eqref{eq:ratios} 
and that~$q(t) =e^{-t^{r-1}}$ by~\eqref{eq:APPrDefs}, 
using $s_2(t_i)/\cpD=(r-1) t_i^{r-2}q(t_i) = -q'(t_i)$ and $\cpD/N=1/M$ it follows that 
\begin{equation}
\begin{split}
\label{eq:expectedQKbound}
\E(\Delta \cQ_{K}(i)\mid\mathcal{F}_i)
=& -\Big(1\pm 4\err(t_i)\Big)\Big(s_2(t_i)\pm 2D^{\frac{1}{r-1}-\delta}\Big)\frac{k}{N}\\
=& \frac{kq'(t_i)}{ M}\pm \Big(4(r-1)t_i^{r-2}\cdot q(t_i) \err(t_i)+ 4D^{-\delta}\Big) \frac{k}{M}.
\end{split}
\end{equation}

In preparation for the upcoming supermartingale arguments, 
we now show that the expected values of the one-step changes~$\Delta X^{\sigma}_K(i):= |X^{\sigma}_K(i+1)|-|X^{\sigma}_K(i)|$ of the auxiliary variables are negative, 
again assuming that $0 \le i < m$ and $\cGo_{\le i}$ hold. 
Set $\cf(t):=q(t)\err(t)$. 
Recalling the shorthand~$t_i=i/M$ and the definition~\eqref{eq:def:X_ipm} of~$X^{\sigma}_K(i)$, 	
by applying Taylor's theorem with remainder to the functions~$q$ and~$f$, it follows~that 
	\begin{equation}\label{eq:secondderivative}
	\begin{split}
\E( \Delta X^{\sigma}_K(i)\mid \mathcal{F}_i) 
& =  \sigma \Bigsqpar{\E(\Delta \cQ_{K}(i)\mid\mathcal{F}_i)-k \bigsqpar{q(t_{i+1})-q(t_i)}}- k \bigsqpar{f(t_{i+1})-f(t_i)}\\ 
& =   \sigma \biggsqpar{\E(\Delta \cQ_{K}(i)\mid\mathcal{F}_i)-\frac{kq'(t_i)}{M}}- \frac{k \cf'(t_i)}{M}+ O\biggpar{\max_{0\le t\le t_\tm} \frac{k\big(|q''(t)|+ |f''(t)|\big)}{M^2}}.
	\end{split}
	\end{equation}
Using~\eqref{eq:expectedQKbound} we see that in~\eqref{eq:secondderivative} the main $kq'(t_i)/M$ term cancels up to second order terms. 
In the following we shall show that the main error term $-k\cf'(t_i)/M$ is large enough to make the expected change~\eqref{eq:secondderivative} negative. 
Indeed, noting~$f(t) = q(t)e(t)  \ge D^{-\delta}$, we have 
\[
f'(t_i) = \bigpar{5 + 4(r-1) t_i^{r-2}} q(t_i)e(t_i)  \ge 5D^{-\delta} + 4(r-1)t_i^{r-2} \cdot q(t_i)e(t_i)  .
\]
Furthermore, ${D=\Theta(N)}$ and ${\delta \le 1/(2r)}$ imply~${M=N/\cpD \gg D^{2\delta}}$. 
Recalling that~${q(t) \le 1}$, ${f(t) \le e(t) \ll 1}$ and~${0 \le t \le D^{o(1)}}$, see~\eqref{eq:ratios}, 
it then routinely follows~that 
\[
\frac{|q''(t)|+ |f''(t)|}{M} \le \frac{O\bigpar{\sum_{0 \le j \le 2r}t^j} \cdot \bigsqpar{q(t)+f(t)}}{M} \le \frac{D^{o(1)}}{D^{2\delta}} \ll D^{-\delta} .
\]
Inserting these estimates and~\eqref{eq:expectedQKbound} into the expected one-step changes~\eqref{eq:secondderivative} of~$X^{\sigma}_K(i)$, 
it follows that 
\begin{equation}\label{eq:supermartingale}
\E (\Delta X^{\sigma}_K(i)\mid\cF_i) \le - \bigpar{1-o(1)}kD^{-\delta}/M < 0 . 
\end{equation}

\subsubsection{Bounds on the one-step changes}\label{subsubsec:onestepchange}
We next bound the expected values of the one-step changes~$|\Delta\cQ_K(i)| = \bigabs{|\cQ_K(i+1)|-|\cQ_K(i)|}$, 
assuming that $0\le i < \tm$ and $\cGo_{\le i}$ hold.
Since~$|\cQ_K(i)|$ is step-wise decreasing, 
by combining~$s_2(t_i) \le r\cpD t_i^{r-2}$ with the first estimate of the expected one-step changes~\eqref{eq:expectedQKbound}, 
using~$\err(t_i)=o(1)$ and~$0 \le t_i \le D^{o(1)}$ it follows that 
\begin{equation}\label{eq:EDeltaQK}
\E (|\Delta \cQ_K(i)| \mid \cF_i) =-\E(\Delta \cQ_K(i)\mid\mathcal{F}_i) \le O\Bigpar{\cpD t_i^{r-2} +D^{\frac{1}{r-1}-\delta}} \cdot \frac{k}{N}
\ll kD^{\frac{1}{r-1}+\delta/2}/N . 
\end{equation}

Turning to the worst-case one-step changes of~$|\cQ_K(i)|$, we introduce the auxiliary event
	\begin{equation}\label{eq:d2ivinK}
\mathcal{N}_{\le j}:=\Big\{	\max_{\cv\in \cV(i)}|\NI_{\cv}(i)\cap K|\le D^{\frac{1}{r-1}-3\delta} \text{ for all $K\in \cA_{N,k}$ and $0 \le i \le j$}   \Big\}.
\end{equation}
Recalling the reasoning leading to~\eqref{eq:expectedQKbound:0}, the crux is that when~$\cN_{\le i}$ holds, then we have 
\begin{equation}\label{eq:absoluteQKbound}
	|\Delta \cQ_K(i)|\le 1+\max_{\cv\in \cV(i)}|\NI_{\cv}(i)\cap K|\le 2 D^{\frac{1}{r-1}-3\delta}.  
\end{equation}
We now claim that the auxiliary event $\mathcal{N}_{\le m}$ typically holds, i.e., more precisely that
\begin{equation}\label{eq:Nm}
\cPr\bigpar{\neg\mathcal{N}_{\le m} \text{ and } \mathcal{S}_{\le \tm}} \le \exp\bigpar{-N^{\Omega(1)}}.
\end{equation}
Turning to the proof details, with an eye on~$|\NI_{\cv}(i)\cap K|$ let 
\begin{equation}\label{eq:def:IKx}
\mathcal{I}=\mathcal{I}(K,\cv)  :=  \bigcpar{W \: : \: |W|=r-2,\; W\cup \{\cv,\cu\} \in\cA_{N,r}\text{ for some $\cu\in K$} } .
\end{equation}
Note that $|\mathcal{I}|\le kr^2$, as there are at most~$r^2$ many~$r$-APs containing two distinct numbers~$\{\cv,\cu\}$. 
Let 
	\begin{equation}\label{eq:NKv} 
	N_{K,\cv}:=\sum_{W\in\mathcal{I}}Y_W \quad \text{ with } \quad Y_W := \indic{W\subseteq I(m)\text{ and $\mathcal{S}_{\le \tm}$}} .
\end{equation}
Since $\{x\}\cup W$ contains $r-1 \ge 2$ elements, by similar reasoning as for $|\mathcal{I}|$ it follows that 
	\begin{equation}\label{eq:NKvUpper} 
	\max_{0 \le i \le m}|\NI_\cv(i)\cap K|\cdot\indic{\mathcal{S}_{\le \tm}}\le N_{K,\cv}\cdot r^2.
\end{equation}
We shall bound $N_{K,\cv}$ via the following Chernoff-type upper tail estimate for combinatorial random variables with `controlled dependencies', 
which is a convenient corollary of~\cite[Theorem~7 and Remarks~9--10]{warnke2020missing}. 
\begin{lemma}\label{lemma:UT}
	Let $(Y_{\alpha})_{\alpha \in \mathcal{I}}$ be a finite family of variables with $Y_{\alpha}\in [0,1]$ and 
	$\sum_{\alpha\in\mathcal{I}}\lambda_\alpha\le \mu$, where $(\lambda_\alpha)_{\alpha\in\mathcal{I}}$ satisfies
	$\E (\prod_{i\in [s]}Y_{\alpha_i})\le \prod_{i\in [s]}\lambda_{\alpha_i}$ for all $(\alpha_1,\dots,\alpha_s)\in\mathcal{I}^s$ with $\alpha_i\cap\alpha_j=\emptyset$ for~$i\neq j$. 
	Set~$Y := \sum_{\alpha\in\mathcal{I}}Y_\alpha$. 
	If 
	$\max_{\alpha \in \mathcal{I}}|\set{\beta \in \mathcal{I}: \alpha \cap \beta \neq \emptyset}| \leq C$,
then $\cPr(Y \ge z) \le (e\mu/z)^{z/C}$ for all $z > \mu$.
\end{lemma}
\noindent 
With an eye on~$\E N_{K,\cv}$, we first record the basic observation that when~$\mathcal{S}_{\le m}$ holds, 
then in every step~$i \le m$ there are at least~$|\cV(i)| \ge |\cV(m)| \gg ND^{-\delta/4}$ available numbers, say.  
For any set~$U$ of numbers from $\mathbb{Z}/N\mathbb{Z}$ a straightforward adaptation of the proof of~\cite[Lemma~4.1]{bohman2010early} 
(which proceeds by taking a union bound over all possible steps where the numbers of~$U$ could appear) 
then ensures that 
\begin{equation}\label{eq:def:P}
\cPr\bigpar{U\subseteq I(m) \text{ and }\mathcal{S}_{\le \tm}} \le m^{|U|} \cdot \Bigpar{\frac{1}{ND^{-\delta/4}}}^{|U|} \le \pi^{|U|} \quad \text{ with } \quad \pi := D^{-\frac{1}{r-1}+\delta/2},
\end{equation}
where we used that $m/N = \cmD (\log N)^{O(1)} \ll D^{-\frac{1}{r-1}+\delta/4}$, say. 
In particular, for any sequence of sets $(W_1,\dots,W_s)\in\mathcal{I}^s$ satisfying~$W_i\cap W_j=\emptyset$ for~$i\neq j$, 
using the definition~\eqref{eq:NKv} of~$Y_{W}$ and~\eqref{eq:def:P} it follows that  
\[
\E \Bigpar{\prod_{i\in [s]}Y_{W_i}}
= \cPr\Bigpar{\bigcup_{i \in [s]} W_i \subseteq I(m) \text{ and }\mathcal{S}_{\le \tm}} 
\le \pi^{s(r-2)}=\prod_{i\in [s]}\lambda_{W_i}  \quad \text{ with } \quad \lambda_W:=\pi^{r-2}. \]
Furthermore, combining $|\mathcal{I}|\le kr^2$ with~\eqref{eq:def:P} and the definition~\eqref{eq:def:k} of~$k$, it also follows that
\[
	 \sum_{W\in\mathcal{I}}\lambda_W\le kr^2\cdot \pi^{r-2}=9\cxi^{-1}r^2(\log n)^{1-\frac{1}{r-1}}D^{\frac{3-r}{r-1}}D^{(r-2)\delta/2}\ll \cpD D^{-4\delta}=:\mu.  
\]
To estimate the associated $C$-parameter of \refL{lemma:UT}, note that any set~$W\in\mathcal{I}$ satisfies 
 \[ \bigabs{\{W'\in\mathcal{I} \: : \: W\cap W'\neq \emptyset  \}} \le \sum_{ w\in W }\sum_{A\in\cA_{N,r}:\{ x,w\}\subseteq A   }2^{|A|}\le  r\cdot r^2\cdot 2^r=:C.   \]
Using inequality~\eqref{eq:NKvUpper}, 
by invoking~\refL{lemma:UT} with $z:=\mu D^{\delta}/r^2\ge D^{\Omega(1)}$ it follows that 
\begin{equation}\label{eq:Nm:UB}
\cPr\Bigpar{\max_{0 \le i \le m}|\NI_x(i)\cap K| \ge \cpD D^{-3\delta} \text{ and } \mathcal{S}_{\le \tm}} 
\le \cPr(N_{K,\cv}\ge z)\le (e\mu/z)^{z/C}  
\le \exp\bigpar{-N^{\Omega(1)}}.  
\end{equation}
Taking a union bound over all of the at most $N \cdot N^2 = N^{O(1)}$ possible pairs~$(x,K)$  
then establishes the claimed inequality~\eqref{eq:Nm}.

\subsubsection{Supermartingale arguments}\label{subsubsec:supermartingale} 
We are now ready to prove~$\cPr(\neg \cGo_{\le \tm}) \le \exp(-N^{\Omega(1)})$, by showing that $X_K^{\sigma}(i) \ge 0$ is extremely unlikely. 
Here our main probabilistic tool is the following supermartingale inequality~\cite[Lemma~19]{guo2020prague}, 
which allows us to exploit that~$X_K^{\sigma}(i)$ is defined~\eqref{eq:def:X_ipm} as the sum of a random variable and a deterministic function. 
\begin{lemma}\label{lem:supermartingale}
Let~$(S_i)_{i \ge 0}$ be a supermartingale adapted to the filtration $(\cF_i)_{i\ge 0}$. 
Assume that~${S_i=X_i+D_i}$, where $X_i$ is~$\cF_i$-measurable and $D_i$ is~$\cF_{\max\{i-1,0\}}$-measurable.
Writing ${\Delta X_i:=X_{i+1}-X_i}$, assume that ${\max_{i \ge 0}|\Delta X_i| \le C}$ and ${\sum_{i \ge 0} \E(|\Delta X_i| \mid \cF_{i}) \le V}$.  
Then, for all~$z >0$, 
\begin{equation}
\label{eq:superm}
\cPr\bigpar{S_i \ge S_0 + z \text{ for some $i \ge 0$}}
\: \le \:
\exp \biggpar{-\frac{z^2}{2C(V+z)}}. 
\end{equation}
\end{lemma}
Turning to the details, we define the stopping time~$T$ as the minimum of~$\tm$ and the first step $i\ge 0$ where the `good' event~$\cGo_{\le i} \cap \mathcal{N}_{\le i}$ fails. 
For brevity, set $i\wedge T:=\min\{i, T\}$.
Recalling the definition~\eqref{eq:sumGm:bound} of the event~$\cGo_{\le \tm}=\mathcal{S}_{\le m} \cap \mathcal{K}_{\le m}$, 
by the discussion below~\eqref{eq:def:X_ipm} it follows that
\begin{equation}\label{eq:sumGm}
\begin{split}
\cPr(\neg \cGo_{\le \tm}) \: 
&=\: \cPr(\neg\mathcal{S}_{\le m} \cup \neg\mathcal{N}_{\le m}) +\cPr(\neg\mathcal{K}_{\le m} \cap \mathcal{S}_{\le m} \cap  \mathcal{N}_{\le m})\\
\: &\le \: \cPr(\neg\mathcal{S}_{\le m}) + \cPr(\neg\mathcal{N}_{\le m}  \text{ and }  \mathcal{S}_{\le m})
+\sum_{\sigma\in\{+,-\}}\sum_{K\in\cA_{N,k}} \cPr\Big(X^\sigma_K(i\wedge T)\ge 0 \text{ for some $i\ge 0$}\Big).
\end{split}
\end{equation}
For any~$K\in\cA_{N,k}$, we initially have $\cQ_K(0)=|K|=k$. 
By definition~\eqref{eq:def:X_ipm} of~$X^{\sigma}_K(i)$ we thus have 
\begin{equation}\label{eq:superm:init}
X^{\sigma}_K(0\wedge T)=X^{\sigma}_K(0) = \sigma[|\cQ_K(0)| - k] - ke(0) = - kD^{-\delta} . 
\end{equation}
Note that the estimates in Sections~\ref{subsubsec:super}--\ref{subsubsec:onestepchange} apply for $0\le i\le T-1$ (since then ${0\le i\le \tm-1}$ and~${\cGo_{\le i}\cap\mathcal{N}_{\le i}}$ hold), 
which in particular implies that the expected one-step changes of~$X^\sigma_K(i)$ satisfy~\eqref{eq:supermartingale}, 
and that the one-step changes~$|\Delta\cQ_K(i)|$ satisfy the worst case bound~\eqref{eq:absoluteQKbound} and the expectation bound~\eqref{eq:EDeltaQK}. 
Recalling~\eqref{eq:superm:init}, the stopped sequence $S_i:=X^\sigma_K(i\wedge T)$ thus is a supermartingale with $S_0=- kD^{-\delta}$, 
to which~\refL{lem:supermartingale} can be applied with $X_i=\sigma|\cQ_K(i\wedge T)|$, $C= O(D^{\frac{1}{r-1}-3\delta})$ and $V=m\cdot k D^{\frac{1}{r-1}+\delta/2}/N= O(k  D^{2\delta/3})$. 
Invoking inequality~\eqref{eq:superm} with $z=kD^{-\delta}$, 
using the definition~\eqref{eq:def:k} of~$k$ and~$D = \Theta(N)$ it follows that 
\begin{equation}\label{eq:superm:X}
\begin{split}
\cPr\Big(X^{\sigma}_K(i\wedge T)\ge 0\text{ for some $i\ge 0$}\Big)
\le \exp\bigpar{-\Omega\bigpar{kD^{\delta/3}/D^{\frac{1}{r-1}}}}
\le \exp\bigpar{-N^{\Omega(1)}}.
\end{split}
\end{equation}
Inserting~\eqref{eq:superm:X} and $|\cA_{N,k}| \le N^{2}$ into~\eqref{eq:sumGm}, 
then $\cPr(\neg \cGo_{\le \tm}) \le \exp(-N^{\Omega(1)})$ follows from~\eqref{eq:probSm} and~\eqref{eq:Nm}, 
which completes the proof of inequality~\eqref{eq:sumGm:bound} and thus \refT{thm:AP}, as discussed. 
\noproof

\subsection{Generalization of \refT{thm:AP}}
We close by recording that a minor variant of our proof yields the following generalization of \refT{thm:AP}. 
\begin{theorem}\label{thm:AP:general}
For any fixed $r\ge 3$ and~$c>0$, there are constants~$C,N_0>0$ such that the following holds for any prime number~$N \ge N_0$, 
setting~$k =k_N:= \ceil{C(N/\log N)^{1/(r-1)}\log N}$. 
For any family~$\cB_{N,k}$ of $k$-element subsets of~$\mathbb{Z}/N\mathbb{Z}$ satisfying~$|\cB_{N,k}| \le N^{c}$ 
there exists a set~$I \subseteq \mathbb{Z}/N\mathbb{Z}$ which (i)~is $r$-AP free in~$\mathbb{Z}/N\mathbb{Z}$ 
and (ii)~satisfies~$|I\cap K|\ge 1$ for all~$K \in \cB_{N,k}$. 
\end{theorem}
\begin{proof}[Proof-Sketch]
Our differential equation method based proof of inequality~\eqref{eq:sumGm:bound} immediately carries over to~$\cB_{N,k}$, 
since besides~$|\cA_{N,k}| \le N^{O(1)}$ we did not use any AP-specific properties
(note that in the union bound arguments below~\eqref{eq:Nm:UB} and~\eqref{eq:superm:X} there is plenty of elbow room). 
Furthermore, after changing the constant~$9$ in the definition~\eqref{eq:def:k} of~$k$ to~$4c+1$, say, 
we see that the probability estimate~\eqref{eq:IcapKnonempty} holds with~$N^{-2}$ replaced by~$N^{-c}$, 
which ensures that the union bound argument establishing \refT{thm:AP} carries over to~$\cB_{N,k}$. 
\end{proof}

\bigskip{\noindent\bf Acknowledgements.} 
We thank the referees for helpful suggestions.

\end{document}